\documentclass{article}
\usepackage{amsmath,amsthm,amsopn,amssymb}
\usepackage{bbding}

\allowdisplaybreaks

\numberwithin{equation}{section}
\newtheorem{theorem}{Theorem}[section]
\newtheorem{lemma}{Lemma}[section]
\makeatother

\begin{document}

\title{A fast Newton-Shamanskii iteration for M/G/1-type and GI/M/1-type Markov chains
\thanks{}}
\author{Pei-Chang Guo \thanks{e-mail: gpeichang@126.com} \\
School of Mathematical Sciences, Peking University, Beijing 100871, China}
\date{}
\maketitle
\begin{abstract}
For the nonlinear matrix equations arising in the analysis of M/G/1-type and GI/M/1-type Markov chains, the minimal nonnegative solution $G$ or $R$ can be found by Newton-like methods. Recently a fast Newton's iteration is proposed in \cite{Houdt2}. We apply the Newton-Shamanskii iteration to the equations. Starting with zero initial guess or some other suitable initial guess, the Newton-Shamanskii iteration provides a  monotonically increasing sequence of nonnegative matrices converging to the minimal nonnegative solution. We use the technique in \cite{houdt2} to accelerate the Newton-Shamanskii iteration. Numerical examples illustrate the effectiveness of the Newton-Shamanskii iteration.

\vspace{2mm} \noindent \textbf{Keywords}: Markov chains, Newton-Shamanskii iteration, Minimal nonnegative solution.
\end{abstract}
\section{Introduction}
Some necessary notation for this article is as follows. For any matrix $B=[b_{ij}]\in \mathbb{R}^{n\times n}$, $B\geq 0 ~(B>0)$ if $b_{ij} \geq 0 ~(b_{ij}> 0)$ for all $i,j$; for any matrices $A,B \in \mathbb{R}^{n\times n}$, $A\geq B~ (A > B)$ if $a_{ij}\geq b_{ij}~(a_{ij} > b_{ij})$ for all $i, j$; the vector with all entries one is denoted by e --- i.e. $e=(1, 1, \cdots, 1)^T$; and the identity matrix is denoted by $I$.
An M/G/1-type Markov Chain (MC) is defined by a transition probability matrix of the form
  \begin{eqnarray*}
P=\left[\begin{array}{ccccc}
B_0 &B_1&B_2&B_3& \cdots \\
C &A_1&A_2&A_3& \cdots \\
 &A_0&A_1&A_2& \ddots \\
 & &A_0&A_1& \ddots \\
 0&& &\ddots& \ddots \\
\end{array}\right],
\end{eqnarray*}
while the transition probability matrix of a GI/M/1-type MC is as follows
 \begin{eqnarray*}
P=\left[\begin{array}{ccccc}
B_0 &C& & & 0 \\
B_1 &A_1&A_0& &   \\
B_2 &A_2&A_1&A_0&  \\
B_3 & A_3 &A_2&A_1& \ddots \\
 \vdots&\vdots& \ddots&\ddots& \ddots \\
\end{array}\right],
\end{eqnarray*}
where $B_0\in \mathbb{R}^{m_0\times m_0}$ and $A_1\in \mathbb{R}^{m\times m}$, respectively. $N$ is the smallest index $i$ such that $A_i$, for $i>N$, is (numerically) zero. The steady state probability vector of an M/G/1-type MC, if it exists, can be expressed in terms of a matrix $G$ that is the element-wise minimal nonnegative solution to the nonlinear matrix equation \cite{Neut2}
\begin{equation}\label{mcequation}
G=\sum_{i=0}^N A_iG^i.
\end{equation}
Similarly, for the GI/M/1-type MC a matrix $R$ is of practical interest, which is the element-wise minimal nonnegative solution to the nonlinear matrix equation \cite{Neut3}
\begin{equation}\label{mcequation2}
R=\sum_{i=0}^N R^iA_i.
\end{equation}
It's known that any M/G/1-type MC can be transformed into a GI/M/1-
type MC and vice versa through either the Ramaswami \cite{Ramas2} or Bright \cite{Houdt1}
dual, and the $G (R)$ matrix can be obtained directly in terms of the $R
(G)$ matrix of the dual chain. The drift of the chain is defined by
\begin{equation}\label{mcrho}
    \rho= p^T\beta,
\end{equation}
where $p$ is the stationary probability
vector of the irreducible stochastic matrix $A=\sum_{i=0}^N A_i$, $\beta=\sum_{i=1}^N iA_ie$.  The MC is positive recurrent if $\rho <1$, null recurrent if $\rho =1$ and transient if $\rho > 1$ --- and throughout this article it is assumed that $\rho \neq1$.

Available algorithms for finding the minimal nonnegative solution to Eq. \eqref{mcequation} include functional iterations \cite{Neut2}, pointwise cyclic reduction (CR) \cite{bini1}, the invariant subspace approach (IS) \cite{Akar}, the Ramaswami reduction (RR) \cite{bini2},and the Newton iteration (NI) \cite{Latou,Neut1,Ramas1,Houdt2}. For the detailed comparison of these algorithms, we refer the readers to \cite{Houdt2} and the references therein. Recently, a fast Newton's iteration is proposed in \cite{Houdt2} and results in substantial improvement on CPU time compared with its predecessors. From numerical experience, the fast Newton's iteration is a very competitive algorithm.

In this paper, the Newton-Shamanskii iteration is applied to the Eq. \eqref{mcequation}.
Starting with a suitable initial guess, the sequence generated by the Newton-Shamanskii iteration is monotonically increasing and converges to the minimal nonnegative solution of Eq. \eqref{mcequation}. Similar with Newton's iteration, equation involved in the Newton-Shamanskii iteration step is also a linear equation of the form $\sum_{j=0}^{N-1} B_jXC^j=E$, which can be solved by a Schur-decomposition method. The Newton-Shamanskii iteration differs from Newton's iteration as the Fr\'{e}chet derivative is not updated at each iteration, therefore the special coefficient matrix structure form can be reused.

The paper is organized as follows. The Newton-Shamanskii iteration and its accelerated iterative procedure using a Schur-decomposition method are given in Section 2. Then M/G/1-type MCs with low-rank downward transitions and low-rank local and upward transitions are considered in  Section 3 and  Section 4, respectively. Numerical results in Section 5 show that the fast Newton-Shamanskii iteration can be more efficient than the fast Newton's iteration proposed in \cite{Houdt2}.  Final conclusions are presented in Section 6.

\section{Newton-Shamanskii Iteration}
In this section we present the Newton-Shamanskii iteration for the Eq. \eqref{mcequation}.
First we rewrite \eqref{mcequation} as
\begin{equation}\label{mcnme}
    \mathcal{G}(X)=\sum_{v=0}^N A_vX^v-X=0
\end{equation}
The function $\mathcal{G}$ is a mapping from $\mathbb{R}^{m\times m}$ into
itself and the Fr\'{e}chet derivative of $\mathcal{G}$ at $X$ is a linear map $\mathcal{G}^{'}_X : \mathbb{R}^{m\times m}\rightarrow \mathbb{R}^{m\times m}$ given by
\begin{equation}\label{mcfdao}
   \mathcal{G}^{'}_X(Z)= \sum_{v=1}^N\sum_{j=0}^{v-1} A_vX^jZX^{v-1-j}-Z.
\end{equation}
The second derivative at $X$, $\mathcal{G}^{''}_X : \mathbb{R}^{m\times m}\rightarrow \mathbb{R}^{m\times m}$, is given by
\begin{equation}\label{mc2fdao}
    \mathcal{G}^{''}_X(Z_1,Z_2)=\sum_{v=2}^N\sum_{j=0}^{v-1} A_v(\sum_{i=0}^{j-1}X^iZ_2X^{j-1-i})Z_1X^{v-1-j}+\sum_{v=2}^N\sum_{j=0}^{v-2}A_vX^jZ_1(\sum_{i=0}^{v-2-j}X^iZ_2X^{v-2-j-i}).
\end{equation}

For a given initial guess $G_{0,0}$, the Newton-Shamanskii iteration for the solution of $\mathcal{G}(x) = 0$ is as follows:

 for $k=0,1,\cdots$
\begin{eqnarray}
\label{mclikea}   \mathcal{G}^{'}_{G_{k,0}}X_{k,s-1}&=&-\mathcal{G}({G_{k,s-1}}),\quad G_{k,s} = G_{k,s-1}+X_{k,s-1}, \quad s=1,2,\cdots , n_k,\\
\label{mclikeb}  G_{k+1}&=&G_{k+1,0}= G_{k,n_k}\;.
\end{eqnarray}
$X_{k,s-1}$ is the solution to
\begin{equation*}
    X_{k,s-1}-\sum_{v=1}^N\sum_{j=0}^{v-1} A_vG_k^jX_{k,s-1}G_k^{v-1-j}=\sum_{v=0}^N A_vG_{k,s-1}^v-G_{k,s-1},
\end{equation*}
which, after rearranging the terms, can be rewritten as
\begin{equation}\label{mcit1}
     X_{k,s-1}-\sum_{j=0}^{N-1}\sum_{v=j+1}^{N} A_vG_k^{v-1-j}X_{k,s-1}G_k^{j}=\sum_{v=0}^N A_vG_{k,s-1}^v-G_{k,s-1}.
\end{equation}
Following the notation of \cite{Houdt2}, we define $S_{k,i}=\sum_{j=i}^NA_jG_k^{j-i}$, then the above equation is
\begin{equation}\label{mcit2}
    (S_{k,1}-I)X_{k,s-1}+\sum_{j=1}^{N-1}S_{k,j+1}X_{k,s-1}G_k^j=G_{k,s-1}-\sum_{v=0}^NA_vG_{k,s-1}^v,
\end{equation}
which is a linear equation of the same form $\sum_{j=0}^{N-1} B_jXC^j=E$ as the Newton's iteration step. It can be solved fast by applying a Schur decomposition on the matrix $C$, which is the $m\times m$ matrix $G_k$ here, and then solving $m$ linear systems with $m$ unknowns and equations. For the detailed description for solving $\sum_{j=0}^{N-1} B_jXC^j=E$, we refer the reader to \cite{Houdt3,Houdt2}. We stress that for Newton-Shamanskii iteration, the coefficient matrices are updated once after every $n_k$ iteration steps and the special coefficient structure can be reused, so the cost per iteration step is reduced significantly.
\section{The Case of Low-Rank Downward Transitions}
When the matrix $A_0$ is of rank $r$, meaning it can be decomposed as $A_0=\widehat{A}_0\Gamma$ with $\widehat{A}_0\in \mathbb{R}^{m\times r}$ and $\Gamma\in \mathbb{R}^{r\times m}$, we refer to the MC as having low-rank downward transitions. If Newton-Shamanskii iteration is applied to this case, all the matrices $X_{k,s-1}$ can be written as $\widehat{X}_{k,s-1}\Gamma$. This can be shown by make induction on the index $s$. $X_{0,0}$ can be written as $\widehat{X}_{0,0}\Gamma$ and we assume that it is true for all $X_{l,j-1}$ for $l=0,\ldots,k$ and $j=1,\ldots,s-1$. Hence $ G_{k,s-1}$ can be written as $ \widehat{G}_{k,s-1}\Gamma$, since $G_{k,s-1}=\sum_{l=0}^{k-1}\sum_{j=1}^{n_l}X_{l,j-1}+\sum_{j=1}^{s-1}X_{k,j-1}
=(\sum_{l=0}^{k-1}\sum_{j=1}^{n_l}\widehat{X}_{l,j-1}+\sum_{j=1}^{s-1}\widehat{X}_{k,j-1})\Gamma$. Then  \eqref{mcit1} can be rewritten  as
\begin{eqnarray*}
\nonumber  X_{k,s-1} &=&\widehat{A}_0\Gamma+ \sum_{j=1}^N A_jG_{k,s-1}^{j-1}\widehat{G}_{k,s-1}\Gamma-\widehat{G}_{k,s-1}\Gamma+\sum_{v=1}^N A_v G_k^{v-1}X_{k,s-1}\\
\nonumber && +\sum_{j=1}^{N-1}\sum_{v=j+1}^{N} A_vG_k^{v-1-j}X_{k,s-1}G_k^{j-1}\widehat{G}_{k}\Gamma,\\
\nonumber   &=&(I-\sum_{v=1}^N A_v G_k^{v-1})^{-1}\\
\nonumber   &&\times (\widehat{A}_0+\sum_{j=1}^N A_jG_{k,s-1}^{j-1}\widehat{G}_{k,s-1}-\widehat{G}_{k,s-1}\sum_{j=1}^{N-1}\sum_{v=j+1}^{N} A_vG_k^{v-1-j}X_{k,s-1}G_k^{j-1}\widehat{G}_{k})\Gamma,
\end{eqnarray*}
therefore $X_{k,s-1}$ can be decomposed as the product of an $m\times r$ matrix $\widehat{X}_{k,s-1}$ and an $r\times m$ matrix $\Gamma$.
The inverse on the right-hand-side exists since $0\leq \sum_{v=1}^N A_v G_k^{v-1}\leq \sum_{v=1}^N A_v G^{v-1}$ and the spectral radius of $\sum_{v=1}^N A_v G^{v-1}$ is strictly than one \cite{bini3}.
Therefore we will concentrate on finding
$\widehat{X}_{k,s-1}$ as the solution to
\begin{eqnarray*}
  \widehat{X}_{k,s-1} &=& \widehat{A}_0+ (\sum_{j=1}^N A_j G_{k,s-1}^{j-1}-I)\widehat{G}_{k,s-1}+\sum_{v=1}^N A_v G_k^{v-1}\widehat{X}_{k,s-1}\\
  & &+\sum_{j=1}^{N-1}\sum_{v=j+1}^{N} A_v G_k^{v-1-j}\widehat{X}_{k,s-1}\Gamma G_k^{j-1}\widehat{G}_{k}\\
   &=&  \widehat{A}_0+(\sum_{j=1}^N A_jG_{k,s-1}^{j-1}-I)\widehat{G}_{k,s-1}+\sum_{j=0}^{N-1}S_{k,j+1}\widehat{X}_{k,s-1}(\Gamma\widehat{G}_{k})^j,
\end{eqnarray*}
which can be rewritten as
\begin{equation}\label{mcitlr}
    (S_{k,1}-I)\widehat{X}_{k,s-1}+\sum_{j=1}^{N-1}S_{k,j+1}\widehat{X}_{k,s-1}(\Gamma\widehat{G}_{k})^j=(I-\sum_{j=1}^N A_jG_{k,s-1}^{j-1})\widehat{G}_{k,s-1}-\widehat{A}_0.
\end{equation}
We can use the Schur decomposition method in \cite{Houdt3,Houdt2} to solve the above equation. Different from the Newton's iteration in \cite{Houdt2}, the special coefficient structure can be reused here, thus saving the overall computational cost. We will repot the numerical performance of the Newton-Shamanskii iteration in Section ?.
\section{The Case of Low-Rank Local and Upward Transitions}
In this section, the case of low-rank local and upward transitions is considered, where the $m\times m $ matrices $\{A_i, 1\leq i \leq N\}$ can be decomposed as $A_i=\Gamma \widehat{A}_i$ with $\Gamma \in \mathbb{R}^{m\times r}$ and $\widehat{A}_i\in \mathbb{R}^{r\times m}$. To exploit low-rank local and upward transitions, we introduce the matrix $U$, which is the generator of the censored Markov chain on level $i$, starting from level $i$, before the first transition on level $i-1$. The following equality holds based on a level crossing argument:
\begin{equation}\label{mcequ}
    U=\sum_{i=1}^N A_i G^{i-1}=  \sum_{i=1}^N A_i ((I-U)^{-1}A_0)^{i-1}.
\end{equation}
For the case of low-rank local and upward transitions, we can rewrite $U$ as
\begin{equation*}
    U= \sum_{i=1}^N A_i ((I-U)^{-1}A_0)^{i-1}=\Gamma[\sum_{i=1}^N \widehat{A}_i ((I-U)^{-1}A_0)^{i-1}]=\Gamma \widehat{U},
\end{equation*}
which means $U$ is of rank $r$, while $G=(I-U)^{-1}A_0$ is generally of rank $m$.

Therefore we find $U$ as the solution to
\begin{equation}\label{mceqf}
    \mathcal{F}(X)=X-\sum_{i=1}^N A_i ((I-X)^{-1}A_0)^{i-1}=0,
\end{equation}
and get $G$ from $G=(I-U)^{-1}A_0$ \cite{Latou2,Houdt2}.
The Newton -Shamanskii iteration step for Eq. \eqref{mceqf} is
as follows:

 for $k=0,1,\cdots$
\begin{eqnarray*}
\nonumber  \mathcal{F}^{'}_{U_k}Y_{k,s-1}&=&-\mathcal{F}({U_{k,s-1}}),\quad U_{k,s} = U_{k,s-1}+Y_{k,s-1}, \quad s=1,2,\cdots , n_k,\\
\nonumber  U_{k+1}&=&U_{k+1,0}= U_{k,n_k}.
\end{eqnarray*}
$Y_{k,s-1}$ is the solution to
\begin{eqnarray}
 \nonumber   Y_{k,s-1}&-&\sum_{i=2}^N A_i\sum_{j=1}^{i-1}((I-U_k)^{-1}A_0)^{j-1}(I-U_k)^{-1}Y_{k,s-1}((I-U_k)^{-1}A_0)^{i-j} \\
 \label{mcitu} &=&\sum_{i=1}^N A_i ((I-U_{k,s-1})^{-1}A_0)^{i-1}-U_{k,s-1}.
\end{eqnarray}
If we define $R_{k,j}=\sum_{i=j+1}^N A_i ((I-U_k)^{-1}A_0)^{i-1-j}(I-U_k)^{-1}$ and rearrange the terms, Eq. \eqref{mcitu} can be rewritten as
\begin{equation*}
    Y_{k,s-1}-\sum_{j=1}^{N-1}R_{k,j} Y_{k,s-1}((I-U_k)^{-1}A_0)^j=\sum_{i=1}^N A_i ((I-U_{k,s-1})^{-1}A_0)^{i-1}-U_{k,s-1},
\end{equation*}
which is of the form $\sum_{j=0}^{N-1} B_jXC^j=E$.
This iteration enables us to exploit low-rank local and upward transitions.
The iterates $U_{k,s}=U_{k,s-1}+Y_{k,s-1}$, where $Y_{k,s-1}$ solves Eq. \eqref{mcitu}, can be rewritten as $ U_{k,s}=\Gamma \widehat{U}_{k,s}$. This can be shown by make induction on the index $s$. It obviously holds for $U_{0,0}$.
Assuming $U_{k,s-1}=\Gamma \widehat{U}_{k,s-1}$, from Eq. \eqref{mcitu} we get
\begin{eqnarray*}
  Y_{k,s-1}&=&\Gamma[\sum_{i=2}^N \widehat{A}_i\sum_{j=1}^{i-1}((I-U_k)^{-1}A_0)^{j-1}(I-U_k)^{-1}Y_{k,s-1}((I-U_k)^{-1}A_0)^{i-j} \\
  &&+\sum_{i=1}^N \widehat{A}_i ((I-U_{k,s-1})^{-1}A_0)^{i-1}-\widehat{U}_{k,s-1}],
\end{eqnarray*}
which tell us that $Y_{k,s-1}$ can be decomposed as $\Gamma \widehat{Y}_{k,s-1}$, and the same holds for $U_{k,s}=U_{k,s-1}+Y_{k,s-1}$.
Therefore from Eq. \eqref{mcitu} we will focus on finding $\widehat{Y}_{k,s-1}$ as the solution to
\begin{eqnarray*}
  \widehat{Y}_{k,s-1}&-&\sum_{i=2}^N \widehat{A}_i\sum_{j=1}^{i-1}((I-U_k)^{-1}A_0)^{j-1}(I-U_k)^{-1}Y_{k,s-1}((I-U_k)^{-1}A_0)^{i-j} \\
  &=&
     \sum_{i=1}^N \widehat{A}_i ((I-U_{k,s-1})^{-1}A_0)^{i-1}-\widehat{U}_{k,s-1}.
\end{eqnarray*}

Defining $\widehat{R}_{k,j}=\sum_{i=j+1}^N \widehat{A}_i ((I-U_k)^{-1}A_0)^{i-1-j}(I-U_k)^{-1}\Gamma$, we can rewrite the above equation  as
\begin{equation}\label{mcitlr2}
    \widehat{Y}_{k,s-1}-\sum_{j=1}^{N-1}\widehat{R}_{k,j} \widehat{Y}_{k,s-1}((I-U_k)^{-1}A_0)^j=\sum_{i=1}^N \widehat{A}_i ((I-U_{k,s-1})^{-1}A_0)^{i-1}-\widehat{U}_{k,s-1},
\end{equation}
which is of the form $\sum_{j=0}^{N-1} B_jXC^j=E$.

\section{Convergence Analysis}

There is monotone convergence when the Newton-Shamanskii method is applied to the Eq.~\eqref{mcequation}.
\subsection{Preliminary}
Let us first recall that a real square matrix $A$ is a $Z$-matrix if
all its off-diagonal elements are nonpositive, and can be
written as $sI-B$ with $B \geq 0$.  Moreover, a $Z$-matrix $A$ is called an $M$-matrix if $s\geq \rho(B)$,
where $\rho(\cdot)$ is the spectral radius; it is a singular $M$-matrix if $s=\rho(B)$, and a
nonsingular $M$-matrix if $s>\rho(B)$.
The following result from Ref.~\cite{varga} is to be exploited.
\begin{lemma}\label{mcyubei1}
For a $Z$-matrix $A$, the following statements are equivalent:
\begin{itemize}
  \item [$(a)$] $A$ is a nonsingular $M$-matrix\;;
  \item [$(b)$] $A^{-1}\geq 0$\;;
  \item [$(c)$] $Av>0$ for some vector $v>0$\;;
  \item [$(d)$] All eigenvalues of $A$ have positive real parts.
\end{itemize}
\end{lemma}
\noindent The following result is also well known~\cite{varga}.
\begin{lemma}\label{mcyubei2}
Let $A$ be a nonsingular $M$-matrix. If $B \geq A$ is a $Z$-matrix, then $B$ is a nonsingular $M$-matrix.  Moreover, $B^{-1}\leq A^{-1}$.
\end{lemma}
\noindent The minimal nonnegative solution $S$ for the Eq.~\eqref{mcequation} may also be recalled --- cf. Ref.~\cite{Latou} for details.
\begin{theorem}\label{mcthm0}
If the rate $\rho$ defined by Eq.~\eqref{mcrho} satisfies $\rho \neq1$, then the matrix
\begin{equation*}
   I- \sum_{v=1}^N\sum_{j=0}^{v-1} (G^{v-1-j})^T \otimes A_vG^j
\end{equation*}
is a nonsingular $M$-matrix.
\end{theorem}
\subsection{Monotone convergence}

The following lemma displays the monotone convergence properties of the Newton iteration for the Eq.~\eqref{mcequation}.
\begin{lemma}\label{mclemm1}
Consider a matrix $X$ such that
\begin{itemize}
  \item [(i)] $\mathcal{G}(X)\geq 0$\;,
  \item [(ii)] $0\leq X\leq G$\;,
  \item [(iii)]$I- \sum_{v=1}^N\sum_{j=0}^{v-1} (X^{v-1-j})^T \otimes A_vX^j$ is a nonsingular $M$-matrix\;.
\end{itemize}
Then the matrix
\begin{equation}\label{mczheng1}
    Y=X-(\mathcal{G}^{'}_{X})^{-1}\mathcal{G}(X)
\end{equation}
is well defined, and
\begin{itemize}
  \item [(a)] $\mathcal{G}(Y)\geq 0$\;,
  \item [(b)] $0\leq X \leq Y\leq G$\;,
  \item [(c)]$I- \sum_{v=1}^N\sum_{j=0}^{v-1} (Y^{v-1-j})^T \otimes A_vY^j$ is a nonsingular $M$-matrix\;.
\end{itemize}
\end{lemma}
\begin{proof}
$\mathcal{G}^{'}_{X}$ is invertible and the matrix $Y$ is well defined, from (iii) and Lemma \ref{mcyubei1}.
Since $$[I- \sum_{v=1}^N\sum_{j=0}^{v-1} (X^{v-1-j})^T \otimes A_vX^j]^{-1}\geq0$$ from (iii) and Lemma \ref{mcyubei1} and $\mathcal{G}(X)\geq 0$, we get that $vec(Y)\geq vec(X)$ and thus $Y\geq X$.
From Eq.~(\ref{mczheng1}) and the Taylor formula, there exists a number $\theta$, $0<\theta_1<1$, such that
   \begin{eqnarray*} \mathcal{G}(Y)&=&\mathcal{G}(X)+\mathcal{G}^{'}_{X}(Y-X)+\frac{1}{2}\mathcal{G}^{''}_X(\theta_1(Y-X),\theta_1(Y-X))\\
       \nonumber &=&\frac{1}{2}\mathcal{G}^{''}_X(\theta_1(Y-X),\theta_1(Y-X))  \\
      \nonumber & \geq &0\;,
     \end{eqnarray*}
so (a) is proven. (b) may be proven as follows.  From
\begin{equation}\label{mctem1}
  0=\mathcal{G}(G)=\mathcal{G}(X)+\mathcal{G}^{'}_{X}(G-X)+\frac{1}{2}\mathcal{G}^{''}_X(\theta_2(G-X),\theta_2(G-X)),
   \end{equation}
  where $0<\theta_2<1$,
we have
 \begin{eqnarray*}
   \nonumber   -\mathcal{G}^{'}_{X}(G-Y)&=&\mathcal{G}^{'}_{X}(Y-X)-\mathcal{G}^{'}_{X}(G-X) \\
      \nonumber &=& -\mathcal{G}(X)-\mathcal{G}^{'}_{X}(G-X)\\
      \nonumber &=&\frac{1}{2}\mathcal{G}^{''}_X(\theta_2(G-X),\theta_2(G-X))\;\\
      \nonumber &\geq& 0,
     \end{eqnarray*}
     where the last inequality is from $G-X \geq 0$ by (ii).
It is notable that $$I- \sum_{v=1}^N\sum_{j=0}^{v-1} (X^{v-1-j})^T \otimes A_vX^j$$ is a nonsingular $M$-matrix, so $vec(G-Y)\geq 0$ from Lemma \ref{mcyubei1} --- i.e. $G-Y\geq 0$. Now $Y\geq X$, so (b) follows. Next we prove (c).
 Since $0\leq Y \leq G$, $$
I- \sum_{v=1}^N\sum_{j=0}^{v-1} (Y^{v-1-j})^T \otimes A_vY^j
\geq
   I- \sum_{v=1}^N\sum_{j=0}^{v-1} (G^{v-1-j})^T \otimes A_vG^j
\;,$$ and $I- \sum_{v=1}^N\sum_{j=0}^{v-1} (G^{v-1-j})^T
 \otimes A_vG^j$ is a nonsingular $M$-matrix.  Consequently from Lemma \ref{mcyubei2},
 $I- \sum_{v=1}^N\sum_{j=0}^{v-1} (Y^{v-1-j})^T \otimes A_vY^j$ is a nonsingular $M$-matrix.
\end{proof}

A generalization of Lemma \ref{mclemm1} provides the theoretical basis for the monotone convergence of the Newton-Shamanskii method for the Eq.~\eqref{mcequation}.
\begin{lemma}\label{mclemm2}
Consider a matrix $X$ such that
\begin{itemize}
  \item [(i)] $\mathcal{G}(X)\geq 0$\;,
  \item [(ii)] $0\leq X\leq G$\;,
  \item [(iii)]$I- \sum_{v=1}^N\sum_{j=0}^{v-1} (X^{v-1-j})^T \otimes A_vX^j$ is a nonsingular $M$-matrix\;.
\end{itemize}
Then for any matrix $Z$ where $0\leq Z\leq X$, the matrix
\begin{equation}\label{qzheng2}
    Y=X-(\mathcal{G}^{'}_{Z})^{-1}\mathcal{G}(X)
\end{equation}
exists such that
\begin{itemize}
  \item [(a)] $\mathcal{G}(Y)\geq 0$\;,
  \item [(b)] $0\leq X\leq Y\leq G$\;,
  \item [(c)]$I- \sum_{v=1}^N\sum_{j=0}^{v-1} (Y^{v-1-j})^T \otimes A_vY^j$ is a nonsingular $M$-matrix\;.
\end{itemize}
\end{lemma}
\begin{proof}
Since $0\leq Z\leq X$, $$
I- \sum_{v=1}^N\sum_{j=0}^{v-1} (Z^{v-1-j})^T \otimes A_vZ^j
\geq
   I- \sum_{v=1}^N\sum_{j=0}^{v-1} (X^{v-1-j})^T \otimes A_vX^j.
\;$$ From (iii) and Lemma \ref{mcyubei2}, $\mathcal{G}^{'}_{Z}$ is invertible and the matrix $Y$ is well defined such that $0\leq X \leq Y.\;$   Let
\begin{equation*}
    \hat{Y}=X-(\mathcal{G}^{'}_{X})^{-1}\mathcal{G}(X)\;,
\end{equation*}
such that $\hat{Y} \geq Y$ from Lemma \ref{mcyubei2}.  As also $\hat{Y}\leq G$ from Lemma \ref{mclemm1}, (b) follows.  Now $$I- \sum_{v=1}^N\sum_{j=0}^{v-1} (\hat{Y}^{v-1-j})^T \otimes A_v\hat{Y}^j$$is a nonsingular $M$-matrix from Lemma \ref{mclemm1} and $\hat{Y} \geq Y$, therefore $I- \sum_{v=1}^N\sum_{j=0}^{v-1} (Y^{v-1-j})^T \otimes A_vY^j $ is a nonsingular $M$-matrix from Lemma \ref{mcyubei2}. Next we show (a) is true. From the Taylor formula, there exists two numbers $\theta_3$ and $\theta_4$, where $0<\theta_3,\theta_4<1$, such that
\begin{eqnarray*} \mathcal{G}(Y)&=&\mathcal{G}(X)+\mathcal{G}^{'}_{X}(Y-X)+\frac{1}{2}\mathcal{G}^{''}_X(\theta_3(Y-X),\theta_3(Y-X))\\
 \nonumber &=&\mathcal{G}(X)+\mathcal{G}^{'}_{Z}(Y-X)+(\mathcal{G}^{'}_{X}-\mathcal{G}^{'}_{Z})(Y-X)+\frac{1}{2}\mathcal{G}^{''}_X(\theta_3(Y-X),\theta_3(Y-X))\\
       \nonumber &=&\mathcal{G}^{''}_Z((Y-X),\theta_4(X-Z))+\frac{1}{2}\mathcal{G}^{''}_X(\theta_3(Y-X),\theta_3(Y-X))  \\
      \nonumber & \geq &0\;,
     \end{eqnarray*}
     where the lat inequality holds since $X-Z\geq 0$ and $Y-X\geq 0.$
\end{proof}
\noindent The monotone convergence result for the Newton-Shamanskii method applied to the Eq.~\eqref{mcequation} follows.
\begin{theorem}\label{mcthm1}
Suppose that a matrix $G_0$ is such that
\begin{itemize}
  \item [(i)] $\mathcal{G}(G_0)\geq 0$\;,
  \item [(ii)] $0\leq G_0\leq G$\;,
  \item [(iii)]$I- \sum_{v=1}^N\sum_{j=0}^{v-1} (G_0^{v-1-j})^T \otimes A_vG_0^j$ is a nonsingular $M$-matrix\,.
\end{itemize}
Then the Newton-Shamanskii algorithm (\ref{mclikea})--(\ref{mclikeb}) generates a sequence $\{G_k\}$
such that $G_k \leq G_{k+1}\leq G$ for all $k \geq 0\,$, and $ \lim_{k \to \infty} G_k=G$.
\end{theorem}
\begin{proof}
The proof is by mathematical induction. From Lemma \ref{mclemm2},
\begin{equation*}
    G_0= G_{0,0}\leq \cdots \leq G_{0,n_0}=G_1 \leq G\;,
\end{equation*}
\begin{equation*}
    \mathcal{G}(G_1)\geq 0\;,
\end{equation*}
and
\begin{equation*}
    I- \sum_{v=1}^N\sum_{j=0}^{v-1} (G_1^{v-1-j})^T \otimes A_v G_1^j
\end{equation*}
is a nonsingular $M$-matrix.
Assuming
\begin{equation*}
    \mathcal{G}(G_i)\geq 0\;,
\end{equation*}
\begin{equation*}
    G_0=G_{0,0}\leq \cdots \leq G_{0,n_0}=G_1 \leq \cdots \leq G_{i-1,n_{i-1}}=G_{i} \leq G\;,
\end{equation*}
and that
$I- \sum_{v=1}^N\sum_{j=0}^{v-1} (G_i^{v-1-j})^T \otimes A_vX_i^j$ is a nonsingular $M$-matrix,
from Lemma \ref{mclemm2}
\begin{equation*}
    \mathcal{G}(G_{i+1})\geq 0\;,
\end{equation*}
\begin{equation*}
    G_i = G_{i,0}\leq \cdots \leq G_{i,n_i}=G_{i+1}\leq G\;,
\end{equation*}
and $I- \sum_{v=1}^N\sum_{j=0}^{v-1} (G_{i+1}^{v-1-j})^T \otimes A_vG_{i+1}^j$ is a nonsingular $M$-matrix.
By induction, the sequence $\{G_k\}$ is therefore monotonically increasing and bounded above by $G$, and so has a limit $G_*$ such that $G_* \leq G$.
Letting $i\rightarrow \infty$ in $G_{i+1}\geq G_{i,1}=G_i-(\mathcal{G}^{'}_{G_i})^{-1}\mathcal{G}({G_i})\geq 0$, it follows that $\mathcal{G}({G_*})=0$.
Consequently, $G_*=G$ since $G_*\leq G$ and $G$ is the minimal nonnegative solution of Eq.~\eqref{mcequation}.
\end{proof}

\section{Numerical Experiments}

So, while more iterations will be needed than for Newton's method, the overall cost of the fast Newton-Shamanskii iteration will be much less.

\section{Conclusions}


\begin{thebibliography}{99}
\bibitem{laub}
J.D. Gardiner, A.J. Laub, J.J. Amato, C.B. Moler. Solution of the Sylvester matrix equation $AXB^T+CXD^T=E$. ACM Trans. Math. Software, 18 (1992), 223每231
\bibitem{Akar}
Akar, N.; Sohraby, K. An invariant subspace approach in M/G/1 and G/M/1 type Markov
chains. Communications in Statistics: Stochastic Models 1997, 13, 381每416.

\bibitem{bini1}
Bini, D.; Meini, B. On the solution of a nonlinear matrix equation arising in queueing
problems. SIAM Journal of Matrix Analysis and Applications 1996, 17, 906每926.
\bibitem{bini2}
Bini, D.; Meini, B.; Ramaswami, V. Analyzing M/G/1 paradigms through QBDs: the role of
the block structure in computing the matrix G. In Latouche, G., Taylor, P. eds. Advances in
Algorithmic Methods for Stochastic Models; pp. 73每86. Notable Publications: Neshanic Station, NJ,
2000.
\bibitem{bini3}
Bini, D.; Latouche, G.; Meini, B. Numerical Methods for Structured Markov Chains; Oxford
University Press: Oxford, UK, 2005.
\bibitem{Neut1}
Neuts, M.F. Moment formulas for the Markov renewal branching process. Advances in Applied
Probability 1976, 8, 690每711.
\bibitem{Neut2}
Neuts, M.F. Structured Stochastic Matrices of M/G/1 Type and Their Applications; Marcel Dekker Inc:
New York, 1989.
\bibitem{Neut3}
Neuts, M.F. Matrix-Geometric Solutions in Stochastic Models; The John Hopkins University Press:
Baltimore, MD, 1981
\bibitem{Orteg}
Ortega, J.M.; Rheinblodt, W.C. Iterative Solution of Nonlinear Equations in Several Variables;
Academic Press: Waltham, MA, 1970.
\bibitem{Ramas1}
Ramaswami, V. Nonlinear matrix equations in applied probability - solution techniques and
open problems. SIAM Review 1988, 30, 256每263.
\bibitem{Ramas2}
Ramaswami, V. A duality theorem for the matrix paradigms in queueing theory.
Communications in Statistics Stochastic Models 1990, 6, 151每161.
\bibitem{Houdt1}
Taylor, P.G.; Van Houdt, B. On the dual relationship between Markov chains of GI/M/1 and
M/G/1 type. Advances in Applied Probability 2010, 42, 210每225.
\bibitem{Houdt3}
P谷rez, J.F.; Van Houdt, B. The M/G/1-type Markov chain with restricted transitions and its
application to queues with batch arrivals. Probability in the Engineering and Informational
Sciences (PEIS) 2011, 25(4), 487每517.
\bibitem{Houdt2}
Juan F. P谷rez , Mikl車s Telek, Benny Van Houdt (2012) A Fast Newton's Iteration for M/G/1-Type and GI/M/1-Type Markov Chains, Stochastic Models, 28:4, 557-583


\bibitem{Latou}
Latouche, G. Newton＊s iteration for non-linear equations in Markov chains. IMA Journal of
Numerical Analysis 1994, 14, 583每598.
\bibitem{Latou2}
Latouche, G.; Ramaswami, V. Introduction to Matrix Analytic Methods in Stochastic Modeling. ASASIAM
Series on Statistics and Applied Probability; SIAM: Philadelphia, PA, 1999.

\bibitem{varga}
{R.~Varga},  Matrix Iterative Analysis, Prentice-Hall (1962).

\end{thebibliography}
\end{document}